\documentclass[a4paper,12pt]{article}
\usepackage{hyperref}
\usepackage{amsmath}
\usepackage{amssymb}
\usepackage{amsthm}
\usepackage{mathrsfs}
\newtheorem{theorem}{Theorem}[section]

\newtheorem{definition}[theorem]{Definition}
\newtheorem{corollary}[theorem]{Corollary}
\newtheorem{remark}[theorem]{Remark}

\newcommand{\EE}{\mathbb{E}}

\newcommand{\NN}{\mathbb{N}}

\newcommand{\PP}{\mathbb{P}}

\newcommand{\DD}{\mathbb{D}}

\newcommand{\D}{{\cal D}}

\newcommand{\F}{{\cal F}}
\newcommand{\G}{{\cal G}}

\newcommand{\X}{{\cal X}}

\begin{document}

\title{\itshape A probabilistic analysis of a discrete-time 
evolution in recombination II. (On partitions)}
\author{Servet Mart{\'i}nez}

\maketitle

\begin{abstract}
We study the discrete-time evolution of a transformation 
on a set of probability measures that is up-dated combining 
independently the marginals on the atoms of  
partitions. This model was recently introduced 
in Baake, Baake and Salamat (Discr. and contin. dynam. syst. 
{\bf 36}, 2016) for continuous-time evolution 
and generalizes 
previous ones based upon dyadic partitions. We associate 
to the discrete-time evolution a natural Markov chain and 
describe its quasi-stationary behavior retrieving 
all the results we recently found for dyadic partitions.
\end{abstract}

\bigskip


\bigskip


\section{ Introduction }
\label{sec0}

Here we study the evolution of the following transformation 
$\Xi$ acting on the set of probability 
measures $\mu$ on a product measurable space $\prod_{i\in I}A_i$, 
$$
\Xi[\mu]=\sum_{\delta\in \G} \rho_\delta \, \bigotimes_{J\in \delta} 
\mu_{J}.
$$
Here $\G$ is a set of partitions of the finite set $I$, 
$\rho=(\rho_\delta: \delta\in \G)$ is a probability vector,
$\mu_J$ is the marginal of $\mu$ on 
$\prod_{i\in J}A_i$, and $\bigotimes _{J\in \delta}\mu_{J}$ 
is the product measure.

\medskip

This transformation was introduced in \cite{bbs}, but in a
continuous-time framework as a generalization of 
dyadic partitions. 
The study of the dynamics $(\Xi^n)$ based on dyadic 
partitions, has served as a model of the genetic 
composition of population under recombination. Most of 
the works devoted to this evolution have considered the 
single cross-over case: 
$I=\{1,..,K\}$ and the dyadic partitions $(J,J^c)$ of the type
$J=\{i: i<j\}$, $J^c=\{i: i\ge j\}$.
We refer to the introductory sections of references \cite{bb}, 
\cite{bvw}, \cite{bbs}, \cite{vwbb} and \cite{uvw} to have a broad 
perspective of the study of $(\Xi^n)$ in relation to sequence 
recombination, as well as a detailed description of the works devoted to
this subject since the pioneer work of H. Geiringer \cite{ge}.

\medskip

Our main results are Theorems \ref{theo0} and Theorem \ref{theo1}
shown in Sections \ref{sec3} and \ref{sec4}, respectively. 
In the first one we associate to the evolution 
$(\Xi^n)$ a natural Markov chain $(Y_n)$ whose transition 
probabilities starting from the coarsest partition, give the 
coefficients $(b_n(\delta))$ of the decomposition $\Xi^n[\mu]
=\sum_{\delta} b_n(\delta) \otimes_{K\in \delta} \mu_K$ written in 
terms of the product of the marginal measures on the atoms of 
partitions $\delta$ on $I$. In our second result, 
which is the main one of this work, we characterize 
the quasi-stationary behavior of the chain $(Y_n)$ before 
attaining the product measure $\otimes_{K\in \D(\G)} \mu_K$,
being $\D(\G)$ the common refinement
of the partitions in $\G$. 
The quasi-stationary results, and their proofs, are 
entirely similar to those  
found in \cite{sm} for the dyadic case. The unique 
additional element 
is that we must prove relations (\ref{im4}) and (\ref{im8})
in Section \ref{sec4} that in the dyadic case were straightforward. 
In \cite{sm} it is given a detailed discussion about this kind of results. 
A main 
interest in quasi-stationarity is because this gives a very precise
information on the deviations of the behavior from the limit measure 
$\otimes_{K\in \D_\G} \mu_K$, and on the other hand because 
the Markov chain $(Y_n)$ has not the 
usual irreducibility conditions, see \cite{cms, ds, pp}. 

\section{ The recombination transformation }
\label{sec1}

First, let us fix some notation on partitions on finite sets. 
Let $I$ be a finite set.
A partition $\delta$ of $I$ is a collection of nonempty sets,
pairwise disjoint and covering $I$. 
We note $\delta=\{L: L\in \delta\}$ and
any of the sets $L$ is called an atom 
of $\delta$. We note by $\DD(I)$ the family of partitions of $I$.

\medskip

For $\delta, \delta'\in \DD(I)$, $\delta'$ is said to be finer than 
$\delta$ or $\delta$ is coarser than $\delta'$, 
we note $\delta\preceq \delta'$,
if every atom of $\delta'$ is contained in an atom of $\delta$. 
The finer partition is $\{\{i\}: i\in I\}$, 
and the coarsest one is $\{I\}$. The common refinement
between two partitions $\delta, \delta'\in \DD(I)$ is noted by
$\delta\vee \delta'$ and its atoms are the nonempty elements
of the family of sets $\{K\cap K': K\in \delta, K\in \delta'\}$. 
One has $\delta\preceq \delta'$ if and only if $\delta\vee 
\delta'=\delta'$.

\medskip

Let $\G$ be a family of partitions of $I$.
We will associate to it the following collection of partitions.
First define $\X_1(\G)=\G$, and by recursion,
$$
\forall\, n\ge 1: \quad
\X_{n+1}(\G)=\{\D\vee \delta: \D\in \G, \delta\in \X_{n}(\G)\}. 
$$ 
Since every $\delta\in \X_{n}(\G)$ satisfies $\D\vee \delta=\delta$
for some element $\D\in \G$, we have
$\X_n\subseteq \X_{n+1}$ for all $n\ge 1$.
This family of sets stabilizes in a finite number of steps, 
that is there exists $n_0\ge 1$ such 
that 
$\X_{n_0+k}(\G)=\X_{n_0}(\G)$ for all $k\ge 0$. Let
$$
\X(\G)=\bigcup_{n\ge 1} \X_n(\G)=\X_{n_0}(\G).
$$
By construction the common refinement of the partitions in $\G$,
$$
\D(\G)=\bigvee_{\D\in \G} \D.
$$
is the finest partition in $\X(\G)$, that is 
$\delta\preceq \D(\G)$ for all $\delta\in \X(\G)$.

\begin{remark}
\label{e13}
$\D(\G)$ is the unique element in $\X(\G)$ that satisfies
$\D(\G)\vee \D=\D(\G)$ for all $\D\in \G$. Moreover, it also holds
$\D(\G)\vee \delta=\D(\G)$ for all $\delta\in \X(\G)$. 
\end{remark}

\medskip

On $\X(\G)$ we define the relation 
\begin{equation}
\label{im5}
\delta\rightarrow \delta'\, \Leftrightarrow \,
\big[\exists \D\in \G: \delta'=\delta\vee \D\big].
\end{equation}
So, $\delta\rightarrow \delta'$ implies $\delta'\succeq \delta$.
Since for every $\delta\in \X(\G_\rho)$ there exists $\D\in \G$
such that $\delta\vee \D=\delta$, we get
\begin{equation}
\label{im6}
\forall \delta\in \X(\G_\rho): \quad \delta\rightarrow \delta. 
\end{equation}

A path between the 
elements $\delta$ and $\delta'$ in $\G$ is a sequence 
$(\delta_k: k=0,..,r)$ in $\G$ such that $\delta_0=\delta$,
$\delta_r=\delta'$ and $\delta_k\rightarrow \delta_{k+1}$
for $k=1,...,r-1$. For every $\delta\in \G\setminus \{I\}$
there exist a path from $\{I\}$ to $\delta$.  
 
\medskip

Now, let us introduce a product measurable space  
and the set of probability measures on it.
Let $(A_i,{\cal B}_i)$, $i\in I$, be a finite collection of measurable 
spaces and let $\prod_{i\in I}A_i$ be a product space endowed with
the product $\sigma-$field $\otimes_{i\in I} {\cal B}_i$.
Denote by ${\cal P}_I$ the set of probability measures
on $\prod_{i\in I}A_i$. 
Let $J\subseteq I$ and ${\cal P}_J$ be the set of probability measures
on $\prod_{i\in J}A_i$.
The marginal $\mu_J\in {\cal P}_J$ of $\mu\in {\cal P}_I$
on $J$ is,
$$
\forall C\in \otimes_{i\in J} {\cal B}_i:\quad 
\mu_J(C)=\mu(C\times \prod_{i\in J^c}A_i)
$$
For $J=I$ we have $\mu_I=\mu$, and    
we put $\mu_\emptyset\equiv 1$ to get
consistency in all the relations where it will appear,
in particular in product measures.

\medskip

For all $J,K\subseteq I$, $J\cap K=\emptyset$,
$\mu_J\in {\cal P}_J$, $\mu_K\in {\cal P}_K$,
let $\mu_J\otimes \mu_K$ be the product measure.
We have that $\otimes$ is commutative and associative, 
$\mu_\emptyset=1$ is the unit element, and 
$\otimes$ is stable under 
restriction, that is,
for all $J, K, M\subseteq I$ with $J\cap K=\emptyset$
and $M\subseteq J\cup K$,
\begin{equation}
\label{eab}
(\mu_J\otimes \mu_{K})_M=\mu_{J\cap M}\otimes \mu_{K\cap M}.
\end{equation}
These are the main properties we require from $\otimes$.

\medskip

From now on, we fix $\rho=(\rho_\delta: \delta\in \DD)$ 
a probability vector, so $\rho_\delta \ge 0$ for $\delta\in \DD$ 
and $\sum_{\delta\in \DD}\rho_\delta=1$. We denote by
$\G_\rho=\{\delta\in \DD: \rho_\delta>0\}$ the support of $\rho$.

\begin{definition}
\label{def1}
Define the following transformation $\Xi: {\cal P}_I\to {\cal P}_I$,
$$
\Xi[\mu]= \sum_{\D\in \G_\rho} 
\rho_\D \, \bigotimes_{J\in \D}\mu_J. \quad\quad\quad \Box
$$
\end{definition}
We note 
$$
D^\rho=\D(\G_\rho)=\bigvee_{\D\in \G_\rho} \D.
$$ 

We claim that
\begin{equation}
\label{cla11}
\mu=\bigotimes_{L\in \D^\rho} \mu_{L}
\hbox{ is a fixed point for $\Xi$}:\;\, \Xi[\mu]=\mu.
\end{equation}
In fact, from $\D^\rho=\D^\rho\vee \D$
for all $\D\in \G_\rho$, we get 
$\mu=\bigotimes_{J\in \D}\mu_J$ for all 
$\D\in \G_\rho$. So, the claim holds.

\section{ The Markov chain}
\label{sec3}

When $\rho_{\{I\}}=1$ we get $\Xi[\mu]=\mu$,
so $\Xi$ is the identity transformation. Then, in the sequel 
we assume 
$$
\rho_{\{I\}}<1 \hbox{ or equivalently } 
\G_\rho\setminus \{I\}\neq \emptyset.
$$ 

Let us define a Markov chain $(Y_n: n\in \NN)$ with values on 
$\X(\G_\rho)$. Its transition matrix
$P=(P_{\delta,\delta'}: \delta, \delta'\in \X(\G_\rho))$ 
is given by
$$
P_{\delta,\delta'}=\sum_{\D\in \G_\rho: \delta\vee \D=\delta'} 
\rho_{\D}.
$$
Note that $P$ is stochastic because $\sum_{\delta'\in 
\X(\G_\rho)}P_{\delta,\delta'}=
\sum_{\D\in \G_\rho} \rho_{\D}=1$.
From definition and (\ref{im5}) we get 
$$
P_{\delta,\delta'}>0 \Leftrightarrow 
\delta\rightarrow \delta'. 
$$
From (\ref{im6})
we have $\delta\rightarrow \delta$, and so
\begin{equation}
\label{im3}
\forall \delta\in \X(\G_\rho): \quad P_{\delta,\delta}>0
\end{equation}
Also note that $P_{\delta,\delta'}>0$ implies $\delta\preceq 
\delta'$ and so when the chain $(Y_n)$ leaves an state 
$\delta$ it does never return to it.

\begin{remark}
\label{cla3}
From Remark \ref{e13} we 
have $\D^\rho\vee \D=\D^\rho$ for all $\D\in \G_\rho$
and so $P_{\D^\rho,\D^\rho}=1$ (which is consistent with (\ref{cla11}). 
Hence, $\D^\rho$ is an absorbing state for the chain $(Y_n)$ and it
is the unique absorbing point for this chain.
\end{remark}

\begin{remark}
\label{remirr}
Since there exists a path $\delta_1=\{I\}\rightarrow...\rightarrow
\delta_k=\delta$ for all $\delta\in \X(\G_\rho)$, $\delta\neq
\{I\}$, this path has positive probability for the Markov chain.
\end{remark}

We claim that $P_{\delta,\delta}$ is strictly increasing 
with $\rightarrow$, that is
\begin{equation}
\label{im9}
\Big[\delta\to \delta', \delta\neq \delta'\Big] \Rightarrow P_{\delta,\delta}<
P_{\delta',\delta'}.
\end{equation}
In fact, every $\D\in \G_\rho$ such that $\delta=\delta\vee \D$ also 
satisfies $\delta'=\delta'\vee \D$. On the other hand there exists $\D_0\in 
\G_\rho$ such that 
$\delta'=\delta\vee \D_0$, and so it also satisfies $\delta'=\delta'\vee \D_0$.
We conclude that $P_{\delta',\delta'}\ge P_{\delta,\delta}+\rho_{\D_0}$,
so (\ref{im9}) follows.

\medskip

We denote by $\PP_\delta$ the law starting from 
$Y_0=\delta$ and by $\PP=\PP_{\{I\}}$ the law of the 
chain starting from $Y_0=\{I\}$. 

\begin{theorem}
\label{theo0}
For all $\mu\in {\cal P}_I$ we have
$$
\Xi^n[\mu]=\sum_{\delta\in \X(\G_\rho)} b_n(\delta) 
\bigotimes_{K\in \delta} \mu_K
$$
with coefficients: 
$$
\forall \delta\in \X(\G_\rho):\quad b_n(\delta)=\PP(Y_n=\delta).
$$
\end{theorem}

\begin{proof}
Let us prove it by induction. 
Let $n=0$. We have $\Xi^0[\mu]=\mu$, so we can take 
$b_0(\{I\})=1=\PP(Y_0=\{I\})$
and $b_0(\delta)=0=\PP(Y_0=\delta)$ for every $\delta\neq \{I\}$,
so the statement holds. 

\medskip

Assume the statement is satisfied for $n$, let
us show it for $n+1$. We have
\begin{eqnarray}
\nonumber
\Xi^{n+1}[\mu]&=&\Xi^{n}[\Xi[\mu]]=
\sum_{\delta\in \X(\G_\rho)} b_n(\delta) \bigotimes_{K\in \delta} {\Xi[\mu]}_K\\
\nonumber
&=&\sum_{\delta\in \X(\G_\rho)} b_n(\delta) \bigotimes_{K\in \delta}
(\sum_{\D\in \G_\rho}\rho_\D \bigotimes_{J\in \D}\mu_J)_K\\
\label{im1}
&=&\sum_{\delta\in \X(\G_\rho)} b_n(\delta) \bigotimes_{K\in \delta}
(\sum_{\D\in \G_\rho}\rho_\D \bigotimes_{J\in \D}\mu_{J\cap K})\\
\nonumber
&=&\sum_{\delta\in  \X(\G_\rho)}\sum_{\D\in \G_\rho} b_n(\delta) \rho_\D 
(\bigotimes_{K\in \delta}\bigotimes_{J\in \D}\mu_{J\cap K})\\
\label{im2}
&=&\sum_{\delta\in  \X(\G_\rho)}\sum_{\D\in \G_\rho} b_n(\delta) \rho_\D
(\bigotimes_{J\cap K\in \D\vee \delta}\mu_{J\cap K}).
\end{eqnarray}
To state (\ref{im1}) we used (\ref{eab}) and in equality
(\ref{im2}) we used $\mu_\emptyset=1$.
Therefore we have the decomposition,
$$
\Xi^{n+1}[\mu]=\sum_{\delta'\in  \X(\G_\rho)} b_{n+1}(\delta')
\bigotimes_{M\in \delta'} \mu_M
$$
with
$$
b_{n+1}(\delta')=\sum_{\delta\in  \X(\G_\rho)}\; \sum_{\D\in \G_\rho: \D\vee 
\delta=\delta'} b_n(\delta) \rho_\D=
\sum_{\delta\in \X(\G_\rho)} b_n(\delta)\left(\sum_{\D\in \G_\rho: \D\vee
\delta=\delta'}  \rho_\D\right).
$$
So, by induction 
we can use that the formula holds for $n$ to get,
$$
b_{n+1}(\delta')=\sum_{\delta\in \X(\G_\rho)} b_n(\delta)
\left(\sum_{\D\in \G: \D\vee
\delta=\delta'}  \rho_\D\right)=
\sum_{\delta\in \X(\G_\rho)} 
\PP(Y_n=\delta)P_{\delta,\delta'}=\PP(Y_{n+1}=\delta').
$$
\end{proof}

\begin{remark}
\label{tree1}
We can expand $\Xi^n$ in terms of 
rooted trees with root $I$ and where to each node it is associated 
an element of $\X(\G_\rho)$, in a similar way as done 
in \cite{sm} for dyadic partitions.
\end{remark}

\section{Quasi-stationary behavior}
\label{sec4}

Let us define the hitting times,
$$
\forall B\subseteq \X(\G_\rho): \quad 
\zeta_B=\inf\{n\ge 0: Y_n\in B\}.
$$
For $\delta\in \X(\G_\rho)$ we simply put
$\zeta_\delta= \zeta_{\{\delta\}}$.
For $\delta=\{I\}$ we have $\PP(\zeta_{\{I\}}=0)=1$. The 
random time for attaining $\D^\rho$ is simply noted,
$$
\zeta=\zeta_{\D^\rho}=\inf\{n\ge 0: Y_n=\D^\rho\}.
$$
Since $\D^\rho$ is an absorbing point, 
then $Y_{\zeta+n}=\D^\rho$ for all $n\ge 0$.
Now, the variables $(Y_n)$ take values in 
$\X(\G_\rho)$, so we can define the sequence of random 
probabilities $(\Xi^n[\mu]=\bigotimes_{K\in Y_n} \mu_K)$. 
Hence, $\Xi^{\zeta+n}[\mu]=\bigotimes_{L\in \D^\rho} \mu_L$
for $n\ge 0$.

\medskip

\begin{theorem}
\label{theo1}
Assume $\rho_I<1$. Then,
\begin{equation}
\label{e29}
\PP(\zeta<\infty)=1.
\end{equation}
Let
$$
\Delta=\{\delta\in \X(\G_\rho): \delta\rightarrow \D^\rho, \delta\neq 
\D^\rho\}.
$$
Define
$$
\eta=\max\{P_{\delta,\delta}: \delta\in \Delta\}
\, \hbox{ and } \,
\F=\{\delta\in \Delta: P_{\delta,\delta}=\eta\}.
$$
Then, $\eta\in (0,1)$ and 
$\PP(\zeta_\F<\infty)>0$. The geometric rate of decay of
$\PP(\zeta>n)$ satisfies,
\begin{equation}
\label{50e}
\lim\limits_{n\to \infty} \eta^{-n} \PP(\zeta\!>\!n)=
\lim\limits_{n\to \infty} \eta^{-n} 
\PP(\zeta\!>\!n, Y_n\!\in \!\F)
=\EE\left(\eta^{-\zeta_\F}, \, \zeta_\F\!<\!\infty \right)
\!\in \! (0,\infty).
\end{equation}
Let
$$
\X(\G_\rho)^*=\X(\G_\rho)\setminus \{\D^\rho\}
\hbox{ and }
P^*=(P_{\delta,\delta'}: \delta,\delta'\in \X(\G_\rho)^*).
$$
The quasi-limiting distribution on  
$\X(\G_\rho)^*$ is given by,
\begin{eqnarray}
\nonumber
\forall \delta\in \F:&{}&
\lim\limits_{n\to \infty} \PP(Y_n=\delta\,| \, \zeta>n)=
\frac{ \EE\left(\eta^{-\zeta_\delta}, \, \zeta_{\delta}<\infty \right)}
{\EE\left(\eta^{-\zeta_\F}, \, \zeta_\F<\infty\right)},\\
\label{e32}
\forall \delta\in \X(\G_\rho)^*\setminus \F:&{}&
\lim\limits_{n\to \infty} \PP(Y_n=\delta\,| \, \zeta>n)=0.
\end{eqnarray}
The following ratio limit relation is satisfied for 
$\delta\in \X(\G_\rho)^*$,
\begin{equation}
\label{50b}
\lim\limits_{n\to \infty}
\frac{\PP_\delta(\zeta>n)}{\PP(\zeta>n)}
=\frac{\EE_\delta(\eta^{-\zeta_\F}, \zeta_\F<\infty)}
{\EE(\eta^{-\zeta_\F}, \zeta_\F<\infty)}.
\end{equation}
Both ratios vanish only when $\PP_\delta(\zeta_\F<\infty)=0$.
The vector 
\begin{equation}
\label{rev1}
\varphi=(\varphi_\delta: \delta\in \X(\G_\rho)^*) \hbox{ with }
\varphi_\delta=\EE_\delta(\eta^{-\zeta_\F}, \zeta_\F<\infty),
\end{equation}
is a right eigenvector of $P^*$ with eigenvalue $\eta$.
\end{theorem}

\noindent {\it Proof}.
It is obvious that $\eta>0$ and from Remark (\ref{cla3}) 
we have $\eta<1$. For $\delta\in \F$ we have that
$P_{\delta,\delta}>0$ (see (\ref{im3})) and $P_{\delta,\D^\rho}>0$
because $\delta\in \Delta$. Let us prove that,
\begin{equation}
\label{im4}
\forall \delta\in \F: \quad
P_{\delta,\delta}+P_{\delta,\D^\rho}=1.
\end{equation}
Assume $P_{\delta,\delta'}>0$ for some $\delta'$ different 
from $\delta$ and $\D^\rho$. So, there exists $\D_0\in \G_\rho$
such that $\delta\vee \D_0=\delta'$. Now, for any $\D\in \G_\rho$
such that $\delta\vee \D=\D^\rho$ we also have 
$\delta'\vee \D=\D^\rho$. We deduce that $\delta'\in \Delta$
and that
$P_{\delta',\delta'}\ge P_{\delta,\delta}+\rho_{\D_0}$. 
Hence, $\eta\ge P_{\delta,\delta}+\rho_{\D_0}$, which contradicts
$\delta\in \F$. This shows (\ref{im4}).
Note that (\ref{im4}) can be written,
$$
\forall \, \delta\in \F, \,:\quad
\delta\rightarrow \delta' \Leftrightarrow \,
\big[\, \delta'=\delta \vee \delta'=\D^\rho\big].
$$

Define,
$$
\beta_0=\max\{P_{\delta,\delta}: \delta\in \X(\G_\rho),
\delta\neq \D^\rho,
\delta\!\not\in \F\}.
$$
Let us prove
\begin{equation}
\label{im8}
\beta_0<\eta.
\end{equation}
If $\delta\in \Delta\setminus \F$, by definition of $\F$ we 
get $P_{\delta,\delta}<\eta$. Let $\delta\not\in \Delta$. 
It is easy to see that there exists a path 
$\delta=\delta_0\rightarrow \delta_1\rightarrow...\rightarrow \delta_r$
for some $\delta_r\in \Delta$ and with all $(\delta_k: k=0,..,r)$ 
different among them. From (\ref{im9}), $P_{\delta_k,\delta_k}$ 
is strictly increasing
with $k$ and so $P_{\delta,\delta}<P_{\delta_r,\delta_r}$. Since 
$P_{\delta_r,\delta_r}<\eta$, relation (\ref{im8}) follows.

\medskip

Let us show (\ref{e29}). As already noted, when $(Y_n)$ exits 
from some state it does never 
return to it. This fact together with inequality 
$P_{\delta,\delta}<1$ for $\delta\neq \D^\rho$, give
$$ 
\forall \delta\in \X(\G_\rho), \delta\neq \D^\rho: \quad 
\PP(\#\{n: Y_n=\delta\}<\infty)=1.
$$
So, since $\D^\rho$ is an absorbing state we get (\ref{e29}): 
$\PP(\zeta<\infty)=\PP(\exists n: Y_n=\D^\rho)=1$.

\medskip

On the other hand, the existence of paths from $\{I\}$ to $\F$ with 
positive probability gives $\PP(\zeta_\F<\infty)>0$.

\medskip

Let us now turn to the proof of relations (\ref{50e}), (\ref{e32}) 
and (\ref{50b}). From (\ref{im4}) we get,
$$
\forall \delta^*\in \F, \,n\ge 0:\quad  
\PP_{\delta^*}(Y_n=\delta^*)=\eta^n.
$$

We have
\begin{equation}
\label{e34}
\PP(\zeta>n)=\PP(\zeta>n, Y_n\not\in \F)+
\PP(\zeta>n, Y_n\in \F).
\end{equation}
Since there exists paths of positive probability from $\{I\}$ 
to $\delta\in \X(\G_\rho)$, $\delta\neq \{I\}$, 
we obtain the existence of $k_0\ge 1$ such that 
$$
\forall \, \delta^*\in \F:\quad \PP(\zeta_{\delta^*} \le k_0)>0.
$$
Define $\alpha(\F):=
\min\{\PP(\zeta_{\delta^*} \le k_0): \delta^*\in \F\}$ which is
strictly positive. From the Markov property we get for 
all $\delta^*\in \F$,
\begin{eqnarray}
\label{e35}
\PP(\zeta\!>\!n)&\ge& \sum_{j=1}^{k_0}
\PP(\zeta_{\delta^*}\!=\!j, \zeta\!>\!n)
\ge \sum_{j=1}^{k_0}\PP(\zeta_{\delta^*}\!=\!j)
\PP_{\delta^*}(\zeta\!>\! n\!-\!j)\\
\nonumber
&\ge & 
\sum_{j=1}^{k_0}\PP(\zeta_{\delta^*}\!=\!j)
\PP_{\delta^*}(Y_{n-j}\!=\!\delta^*)
\ge \sum_{j=1}^{k_0}\PP(\zeta_{\delta^*}\!=\!j) \eta^{n-j}
\ge \alpha(\F) \eta^n.
\end{eqnarray}

To analyze the first term at the right hand side of 
equality (\ref{e34}), we will use the following simple result, 
which is proven in detail in Lemma $5.6$ in \cite{sm}. We have,
\begin{equation} 
\label{e36}
\forall\, \theta\!>\!0 \, \exists C'\!=\!C'(\theta):\quad 
\PP(\forall j\!\le \! n: \; Y_j\not\in \F \cup \{\D^\rho\})
\le C'(\beta_0\!+\!\theta)^n.
\end{equation} 
We will always take $\theta>0$ such that 
$\beta_0+\theta<\eta$. Hence, from (\ref{e35}) and (\ref{e36}) 
we find
\begin{equation}
\label{50a}
\PP(Y_n\not\in \F \, | \, \zeta>n)\le 
C'' \left((\beta_0+\theta)/\eta\right)^n\to 0
\hbox{ as } n\to \infty,
\end{equation}
with $C''=C'/\alpha(\F)$. Therefore,
\begin{equation}
\label{50f}
\lim\limits_{n\to \infty}
\PP(Y_n\in \F \, | \, \zeta>n)=1.
\end{equation}
Let us examine the second term at the right hand side 
of equality (\ref{e34}). For every $\delta^*\in \F$ we have
\begin{eqnarray*}
\PP(\zeta>n, Y_n=\delta^*)&=&\sum_{j=1}^n 
\PP(\zeta>n, \zeta_{\delta^*}=j)\\
&=&\sum_{j=1}^n\PP(\zeta_{\delta^*}=j)
\PP_{\delta^*}(\zeta>n-j)\\
&=&\sum_{j=1}^n\PP(\zeta_{\delta^*}=j)\eta^{n-j}
=\eta^n
\left(\sum_{j=1}^n\eta^{-j}\PP(\zeta_{\delta^*}=j)\right).
\end{eqnarray*}
Since 
\begin{eqnarray*}
\PP(\zeta_{\delta^*}=j)&\le& \PP(\zeta_\F=j)\\
&\le&\PP(\forall n\le j-1: \; 
Y_n\not\in \F\cup \{\D^\rho\})
\le C'(\beta_0+\theta)^{j-1},
\end{eqnarray*}
and $\beta_0+\epsilon<\eta$, we get 
$\sum_{j=1}^\infty \eta^{-j}\PP(\zeta_{\delta^*}=j)<\infty$.
Hence, 
\begin{eqnarray}
\label{e40y}
\forall \delta^*\in \F:\; 
\lim\limits_{n\to \infty}\eta^{-n}\PP(\zeta>n, Y_n=\delta^*)&=&
\sum_{j=1}^\infty\eta^{-j}\PP(\zeta_{\delta^*}=j)\\
\nonumber
&=&
\EE\left(\eta^{-\zeta_{\delta^*}}, \zeta_{\delta^*}<\infty \right)<\infty.
\end{eqnarray}
Now, for $\delta^*\in \F$ we have
$$
\zeta_{\delta^*}<\infty \, \Rightarrow \, 
\big[\, \forall \delta'\in \F\setminus \{\delta^*\}: \; 
\zeta_{\delta'}=\infty 
\hbox{ and } \zeta_\F=\zeta_{\delta^*} \, \big].
$$
Then,
$$
\{\zeta_\F=j\}=\bigcup_{\delta^*\in \F} \{\zeta_{\delta^*}=j\}
$$
and the union is disjoint. So,
$\eta^{-\zeta_\F} {\bf 1}_{\zeta_\F<\infty}=
\sum_{\delta^*\in \F}\eta^{-\zeta_{\delta^*}}
{\bf 1}_{\zeta_{\delta^*}<\infty}$.
Hence,
$$
\EE\left(\eta^{-\zeta_\F}, \zeta_\F<\infty\right)=
\sum_{\delta^*\in \F} \EE\left(\eta^{-\zeta_{\delta^*}},
\zeta_{\delta^*} <\infty\right)<\infty.
$$
Then, from (\ref{e40y}), we deduce
\begin{equation}
\label{e40yx}
\lim\limits_{n\to \infty}
\eta^{-n}\PP(\zeta>n, Y_n\in \F)=
\EE\left(\eta^{-\zeta_\F}, \zeta_\F<\infty\right).
\end{equation}
Therefore, relations (\ref{50a}), (\ref{e40y}) and (\ref{e40yx}), 
give (\ref{e32}). 

\medskip

Now, relation (\ref{50e}) is a consequence of relations 
(\ref{50f}) and (\ref{e40yx}) because they imply
\begin{eqnarray*}
\lim\limits_{n\to \infty} \eta^{-n} \PP(\zeta>n)&=&
\lim\limits_{n\to \infty} \eta^{-n} \PP(\zeta>n, Y_n\in 
\F)\\
&=&\EE(\eta^{-\zeta_\F}, \zeta_\F<\infty)\in (0,\infty).
\end{eqnarray*}

\medskip

Let us show (\ref{50b}). First, assume $\delta$ is such that 
$\PP_\delta(\zeta_\F<\infty)>0$.
Since there is a path with
positive probability from $\delta$
to some nonempty subset of $\F$,
a similar proof as the one showing (\ref{50e}) gives 
$$
\lim\limits_{n\to \infty} \eta^{-n}
\PP_\delta(\zeta>n)=
\EE_\delta(\eta^{-\zeta_\F},\zeta_\F<\infty)\in (0,\infty),
$$
so (\ref{50b}) is satisfied.
Now, let $\PP_\delta(\zeta_\F<\infty)=0$. Then, 
$\EE_\delta(\eta^{-\zeta_\F}, \zeta_\F<\infty)=0$
and in (\ref{50b}) we have
${\EE_\delta(\eta^{-\zeta_\F}, \zeta_\F<\infty)}/
{\EE(\eta^{-\zeta_\F}, \zeta_\F<\infty)}=0$. 
We claim that in this case we also have
$\lim\limits_{n\to \infty} \PP_\delta(\zeta>n)/\PP(\zeta>n)=0$.
In fact $\PP_\delta(\zeta_\F<\infty)=0$ implies
\begin{eqnarray*}
(\beta_0+\theta)^{-n}\PP_\delta(\zeta>n)&=&
(\beta_0+\theta)^{-n}\PP_\delta(\zeta>n, \zeta_\F>n)\\
&=&(\beta_0+\theta)^{-n} \PP(\forall j\le n: 
Y_j\not\in (\F\cup \{\D^\rho\})<\infty. 
\end{eqnarray*}
Since $\lim\limits_{n\to \infty}\eta^{-n}\PP(\zeta>n)>0$ 
and $\beta_0+\theta<\eta$, the claim follows and (\ref{50b})
is shown.

\medskip

The last statement to be proven is that $\varphi$ 
defined in (\ref{rev1}) is a right eigenvector of $P^*$ 
with eigenvalue $\eta$. First take $\delta\in \F$. We 
have $\PP_\delta(\zeta_\F=0)=1$ and so
$\EE_\delta(\eta^{-\zeta_\F},\zeta_\F<\infty)=1$.
From (\ref{im4}) and $P_{\delta,\delta}=\eta$ we get
$$
(P^* \varphi)_\delta=
\sum_{\delta':\delta'\neq D^\rho, \delta\to \delta'} P_{\delta,\delta'}
\EE_{\delta'}(\eta^{-\zeta_\F},\zeta_\F<\infty)=\eta
=\eta\, \varphi_\delta.
$$
Now let $\delta$ be such that $\PP_\delta(\zeta_\F<\infty)=0$,
so $\varphi_\delta=0$.
Then $P_{\delta,\delta'}>0$ implies  
$\PP_{\delta'}(\zeta_\F<\infty)=0$ and so
$(P^* \varphi)_\delta=0=\eta\, \varphi_\delta$.

\medskip

Now take $\delta\not\in \F$ with 
$\PP_\delta(\zeta_\F<\infty)>0$. From
the Markov property we get,
\begin{eqnarray*}
\varphi_\delta&=&
\EE_\delta(\eta^{-\zeta_\F},\zeta_\F<\infty)=
\sum_{\delta':\delta'\neq D^\rho, \delta\to \delta'} 
\EE_\delta(\eta^{-\zeta_\F},\zeta_\F<\infty, Y_1=\delta')\\
&=&\sum_{\delta':\delta'\neq D^\rho, \delta\to \delta'} 
P_{\delta,\delta'} \; \eta^{-1}\,
\EE_{\delta'}(\eta^{-\zeta_\F},\zeta_\F<\infty)
=\eta^{-1}\, (P^* \varphi)_\delta.
\end{eqnarray*}
Then, the result is shown, which finishes the proof of the theorem.
$\Box$

\bigskip

 From Theorem \ref{theo1} we will obtain two other results:
the description of the $Q-$process, which in our case 
is the Markov chain that avoids the singleton
$\{\otimes_{L\in \D^\rho}\mu_L\}$, and an explicit class 
of quasi-stationary distributions, that must be compared 
with the irreducible case 
where there is a unique quasi-stationary distribution. 
The $Q-$process was introduced 
in \cite{an} for branching processes, and developments on 
$Q-$processes in other contexts that include finite 
Markov chains, are found in \cite{cms}.

\begin{corollary}
\label{cor1}
For all $\delta_i\in \X(\G_\rho)^*$, $i=1,..,k$, the following limit exists
$$
\lim\limits_{n\to \infty}\PP(Y_i=\delta_i, i=1,..,j \, | \, \zeta>n)
$$ 
and it vanishes if some 
$\delta_i$ satisfies $\PP_{\delta_i}(\zeta_\F<\infty)=0$.

\medskip

Denote 
$$
\partial(\zeta_\F)=\{\delta\in \X(\G_\rho)^*:
\PP_{\delta}(\zeta_\F<\infty)>0\}.
$$ 
Then, the matrix 
$Q=\left(Q_{\delta,\delta'}: \delta,\delta'\in 
\partial(\zeta_\F)\right)$
given by
$$
Q_{\delta,\delta'}=\eta^{-1}\, 
P_{\delta,\delta'}\frac{\EE_{\delta'}(\eta^{\zeta_\F},
\zeta_\F<\infty)}{\EE_\delta(\eta^{\zeta_\F}, \zeta_\F<\infty)},
$$
is an stochastic matrix on 
$\partial(\zeta_\F)$, and it is satisfied
$$
\forall \delta_i\in \partial(\zeta_\F), i=0,..,j: \quad
\lim\limits_{n\to \infty}\PP_{\delta_0}
(Y_i=\delta_i, i=1,..,j \, | \, \zeta>n)=
\prod_{i=0}^{j-1} Q_{\delta_i,\delta_{i+1}}.
$$
So, $Q$ is the transition matrix of the Markov chain that 
never hits $\otimes_{L\in \D^\rho}\mu_L$.
\end{corollary}

\begin{proof}
Let us prove that $Q$ is an stochastic matrix. Let
$\varphi$ be the right eigenvector of $P^*$
with eigenvalue $\eta$  given in (\ref{rev1}). The component
$\varphi_\delta$ vanishes when 
$\PP_{\delta}(\zeta_\F<\infty)=0$. Let $\delta\in 
\partial(\zeta_\F)$. 
We will use that
$P_{\delta,\delta'}=0$ if $\delta\not\rightarrow \delta'$ and that
$$
\PP_{\delta'}(\zeta_\F<\infty)=0 \hbox{ implies } 
\frac{\varphi_{\delta'}}{\varphi_\delta}=
\frac{\EE_{\delta'}(\eta^{\zeta_\F},
\zeta_\F<\infty)}{\EE_\delta(\eta^{\zeta_\F}, \zeta_\F<\infty)}=0.
$$
Then, since $\varphi$ is a right eigenvector
with eigenvalue $\eta$ we get
$$
\sum_{\delta'\in \partial(\zeta_\F)}Q_{\delta,\delta'}=
\eta^{-1}\left(\sum_{\delta'\in \partial(\zeta_\F)}
P_{\delta,\delta'}\frac{\varphi_{\delta'}}{\varphi_{\delta}}\right)
=\eta^{-1} \left( \sum_{\delta'\in \X(\G_\rho)^*}
P_{\delta,\delta'}\frac{\varphi_{\delta'}}{\varphi_{\delta}}\right)=1.
$$
From the Markov property we obtain for $n>j$,
$$
\PP(Y_i=\delta_i, i=1,..,j \, | \, \zeta>n)=
\PP(Y_i=\delta_i, i=1,..,j)\frac{\PP_{\delta_j}(\zeta>n-j)}
{\PP(\zeta>n)},
$$
Now we use the ratio limit result (\ref{50b}).
This limit vanishes if $\PP_{\delta_j}(\zeta_\F<\infty)=0$ and
it also vanishes when $\PP_{\delta_i}(\zeta_\F<\infty)=0$ 
for some $i<j$ because 
$P_{\delta_i,\delta_{i+1}}>0$ implies 
$\PP_{\delta_{i+1}}(\zeta_\F<\infty)=0$.
For $\delta_i\in  \partial(\zeta_\F)$ for $i=0,..,j$, we 
have
\begin{eqnarray}
\nonumber
&{}&\lim\limits_{n\to \infty}\PP_{\delta_0}
(Y_i=\delta_i, i=1,..,j \, | \, \zeta>n)\\
\nonumber
&{}&
=\lim\limits_{n\to \infty}
\PP_{\delta_0}(Y_i=\delta_i, i=1,..,j)
\frac{\PP_{\delta_j}(\zeta>n-j)}
{\PP_{\delta_0}(\zeta>n)}\\
\label{rats}
&{}&=\PP_{\delta_0}(Y_i=\delta_i, i=1,..,j)
\frac{\varphi_{\delta_j}}{\varphi_{\delta_0}}\eta^{-j}
=\prod_{l=0}^{j-1} \left(\eta^{-1}P_{\delta_l, \delta_{l+1}}\,
\frac{\varphi_{\delta_{l+1}}}{\varphi_{\delta_l}}\right).
\end{eqnarray}
In (\ref{rats}) we used 
$\lim\limits_{n\to \infty}\PP(\zeta>n-j)/\PP(\zeta>n)=\eta^{-j}$,
which is a consequence of (\ref{50e}).
Then the result follows.
\end{proof}

\begin{remark}
\label{intf1}
In the above $Q-$process all the states $\F$ are
absorbing states, that is $Q_{\delta^*,\delta^*}=1$ for all 
$\delta^*\in \F$. Hence, once the $Q-$process attains one 
of the states in $\F$ it remains in it forever. 
\end{remark}

Let $\nu=(\nu_\delta: \delta\in \X(\G_\rho)^*)$ be a 
probability measure on $\X(\G_\rho)^*$. 
If necessary, $\nu$ will be identified with its extension
on $\X(\G_\rho)$ with $\nu_{D^\rho}=0$. We say that $\nu$ is 
supported 
by some subset ${\widetilde{\partial}}\subseteq \X(\G_\rho)^*$ if 
$\nu({\widetilde{\partial}})=1$.
We denote by $\nu'$ the row vector associated to $\nu$.

\begin{corollary}
\label{cor2}
Every probability measure $\nu$ on $\X(\G_\rho)^*$ 
supported on $\F$ satisfies 
$\nu'P^*=\eta \, \nu'$ and it is a quasi-stationary 
distribution, that is it satisfies
\begin{equation}
\label{e43}
\forall n\ge 1, \, \forall \delta\in \X(\G_\rho)^*: \quad 
\PP_\nu(Y_n=\delta \, | \, \zeta>n)=\nu_\delta. 
\end{equation}
\end{corollary}

\begin{proof} 
With the above notation and by using (\ref{im4}) we get,
$$
(\nu' P^*)_\delta= P_{\delta,\delta} \, 
\nu_\delta=\eta \, \nu_\delta, 
$$
so $\nu' P^*=\eta \nu'$. By iteration we find 
$\nu' P^{*n}=\eta^n \, \nu'$.
Note that this is equivalent to
$$
(\nu' P^{*n})_\delta=\PP_{\nu}(Y_n=\delta)=\PP_{\nu}(\forall j\le n \; 
Y_j=\delta) =\eta^n \, \nu'_\delta.
$$ 
Now
$$
\PP_\nu(\zeta>n)=
\sum_{\delta\in \F} (\nu' P^{*n})_\delta=\eta^n 
\left(\sum_{\delta\in \F} \nu_\delta\right)=\eta^n.
$$ 
Hence, relation (\ref{e43}) is proven. 
\end{proof}

An analogous results cane stated for positive eigenvectors. Let 
$\widetilde{\partial}\subseteq \F$
be a nonempty set, then the characteristic function 
${\bf 1}_{\widetilde{\partial}}$ is a right eigenvector of $P^*$ 
with eigenvalue $\eta$. 

\bigskip

\noindent{\bf Acknowledgments}. We thank support from the CMM Basal 
CONICYT Project PB-03.

\noindent SERVET MART\'INEZ

\noindent {\it Departamento Ingenier{\'\i}a Matem\'atica and Centro
Modelamiento Matem\'atico, Universidad de Chile,
UMI 2807 CNRS, Casilla 170-3, Correo 3, Santiago, Chile.}
e-mail: smartine@dim.uchile.cl

\label{lastpage}

\end{document}